\documentclass{gtpart}
\usepackage{mathrsfs}
\usepackage{amsfonts}
\usepackage{amsmath}
\usepackage{amsthm}
\usepackage{dsfont}
\usepackage{amssymb,latexsym,amsmath}
\usepackage{graphicx}
\usepackage{float}
\usepackage{caption}
\usepackage{enumerate}
\usepackage{hyperref}

\title{Non L-space integral homology 3-spheres with no nice orderings}
\author{Xinghua Gao}
\givenname{Xinghua}
\surname{Gao}
\email{xgao29@illinois.edu}
\address{Department of Mathematics\\
1409 W. Green St.\\\newline
Urbana, IL 61801\\
United States}



\arxivreference{1609.07663}  
\arxivpassword{Bnu20091625}   

\newtheorem{theorem}{Theorem}

\newtheorem{lemma}{Lemma}

\newtheorem*{remark}{Remark}
\numberwithin{equation}{section}

\begin{document}

\begin{abstract}
This paper gives infinitely many examples of non L-space irreducible integer homology 3-spheres whose fundamental groups do not have nontrivial $\widetilde{PSL_2(\mathbb{R})}$ representations.
\end{abstract}

\maketitle

\section{Introduction}

Before stating the main result, I will review some definitions.
A rational homology 3-sphere $Y$ is called an \emph{L-space} if $\text{rk} \widehat{HF}(Y)=|H_1(Y; \mathbb{Z})|$, i.e. its Heegaard Floer homology is minimal. An L-space does not admit any co-orientable taut foliation by Bowden \cite{1509.07709}, Kazez-Roberts \cite{1404.5919} and Ozsv\'{a}th-Szab\'{o} \cite{taut}.
A nontrivial group $G$ is called left-orderable if there exists a strict total ordering of $G$ invariant under left multiplication.
Boyer, Gordon, and Watson conjectured in \cite{Boyer} that an irreducible rational homology 3-sphere is a non L-space if and only if its fundamental group is left-orderable.
A stronger conjecture states that for an irreducible $\mathbb{Q}$-homology 3-sphere, being a non L-space, having left-orderable fundamental group and admitting a co-orientable taut foliation are the same (see e.g. Culler-Dunfield \cite{CD16}).

To show the fundamental group $\pi_1(Y)$ of a $3$-manifold $Y$ is orderable, it is most common to consider $\widetilde{PSL_2(\mathbb{R})}$ representations of $\pi_1(Y)$. In fact in many cases, $\widetilde{PSL_2(\mathbb{R})}$ representations are sufficient to define an order on $\pi_1(Y)$ \cite{CD16}.
However, Theorem \ref{result} in this paper shows that, even in the case of non L-space integral homology spheres, orders coming from $\widetilde{PSL_2(\mathbb{R})}$ are not enough to prove the conjecture of Boyer, Gordon and Watson.

It is conjectured that any integer homology $3$-sphere different from the $3$-sphere admits an irreducible representation in $SU_2(\mathbb{C})$ (see e.g. Kirby's problem list \cite[Problem 3.105]{Kirby}). Zentner showed that if one enlarges the target group to $SL_2(\mathbb{C})$, then every such integral homology 3 sphere has an irreducible representation \cite{Zentner}. In contrast, I will give examples where there are no irreducible $PSL_2(\mathbb{R})$ representations. Let $\mathcal{M}$ be the manifold $m137$ \cite{m137cencus} and $\mathcal{M}(1,n)$ be the integral homology sphere obtained by $(1,n)$ Dehn fillings on $\mathcal{M}$. 
The main result of this paper states:
\begin{theorem}
\label{result}
For all $n\ll 0$, the manifold $\mathcal{M}(1,n)$ is a hyperbolic integral homology $3$-sphere where
\begin{itemize}
\item[a)] $\pi_1(\mathcal{M}(1,n))$ does not have a nontrivial $\widetilde{PSL_2(\mathbb{R})}$ representation.
\item[b)] $\mathcal{M}(1,n)$ is not an L-space.
\end{itemize}
\end{theorem}
\noindent This means that we can not produce an order on $\pi_1(\mathcal{M}(1,n))$ simply by pulling back the action of $\widetilde{PSL_2(\mathbb{R})}$ on $\mathbb{R}$.

Section 1 is devoted to proving part (a) of Theorem \ref{result}. Let $X_0(\mathcal{M})$ be the component of the $SL_2(\mathbb{C})$ character variety of $\mathcal{M}$ containing the character of an irreducible representation (see Culler-Shalen \cite{shalen} for definition). Here is an outline of the approach. Let $X_{0,\mathbb{R}}(\mathcal{M})$ be the real points of $X_0(M)$. Define $[\rho]\in X_{0,\mathbb{R}}(\mathcal{M})$ and denote by $s$ the trace of $\rho(\lambda)$ where $\lambda$ is the homological longitude of $\mathcal{M}$. The proof is divided into two parts. In the first part, I show that points on the $|s|<2$ components of $X_{0,\mathbb{R}}(\mathcal{M})$ all correspond to $SU_2(\mathbb{C})$ representations while points on the $|s|>2$ components correspond to $SL_2(\mathbb{R})$ representations. In the second part, I show that $SL_2(\mathbb{R})$ representations of $\pi_1(\mathcal{M})$ give rise to no $SL_2(\mathbb{R})$ representations of $\pi_1(\mathcal{M}(1,n))$ when $n\ll 0$. This part of the proof is basically analysing real solutions to the A-polynomial of $\mathcal{M}$ under the relation $\mu\lambda^n=1$ given by $(1,n)$ Dehn filling , where $\mu$ is a choice of meridian of $\partial \mathcal{M}$.

In Section 2, by applying techniques in the paper by Rasmussen, Rasmussen \cite{simple} and Gillespie \cite{gillespie}, I show that none of the $(1,n)$ Dehn fillings on $m137$ is an L-space, completing the proof of Theorem \ref{result}.

\section*{Acknowledgements}
The author was partially supported by NSF grants DMS-1510204, and Campus Research Board grant RB15127. I would like to pay special thanks to my advisor, Nathan Dunfield for suggesting me this problem and offering me extraordinary help. I would also like to thank the referee for detailed and helpful comments and suggestions.

\section{$\widetilde{PSL_2(\mathbb{R})}$ representations}

I will prove Theorem \ref{result} (a) in this section.

SnapPy \cite{SnapPy} gives us the following presentation of the fundamental group of $\mathcal{M}=m137$:
\begin{align*}
\pi_1(\mathcal{M}) =\langle\alpha, \beta\ |\ \alpha^3\beta^2\alpha^{-1}\beta^{-3}\alpha^{-1}\beta^2\rangle.
\end{align*}
The peripheral system of $\mathcal{M}$ can be represented as:
\begin{align*}
\{\mu, \lambda\} = \{\alpha^{-1}\beta^2\alpha^4\beta^2, \alpha^{-1}\beta^{-1}\}=\{\beta^2\lambda^{-1}\beta^{-3}\lambda^{-1}\beta^2, \lambda\}
\end{align*}
where $\lambda$ is the homological longitude and $\mu$ is a choice of meridian.
Then we can rewrite the fundamental group as:
\begin{align}\label{pi1}
\pi_1(\mathcal{M}) =\langle\lambda, \beta\ |\ \beta^{-1}\lambda^{-1}\beta^{-1}\lambda^{-1}\beta^2\lambda=\lambda\beta^{-2}\lambda^{-1}\beta^2 \rangle,
\end{align}
and the meridian becomes $\mu = \beta^2\lambda^{-1}\beta^{-3}\lambda^{-1}\beta^2$ under this presentation.
\begin{remark}
The triangulation of $m137$ we used (included in \cite{anc}) to get these presentations is different from SnapPy's default triangulation. We got it by performing random Pachner moves on the default triangulation in SnapPy. In particular, our notations for longitude and meridian in the peripheral system are meridian and longitude respectively in SnapPy's default notations.
\end{remark}

We will first look at irreducible $SL_2(\mathbb{C})$ representations of the fundamental group of $\mathcal{M}$ before we look at those of Dehn fillings of $\mathcal{M}$.
Denote by $X(\mathcal{M})$ the $SL_2(\mathbb{C})$ character variety of $\mathcal{M}$, that is the Geometric Invariant Theory quotient Hom$(\pi_1(\mathcal{M}),SL_2(\mathbb{C}))//SL_2(\mathbb{C})$. It is an affine variety \cite{shalen}. Suppose $\rho: \pi_1(\mathcal{M})\longrightarrow SL_2(\mathbb{C})$ is a representation of the fundamental group of $\mathcal{M}$.
Recall that a representation $\rho$ of $G$ in $SL_2(\mathbb{C})$ is irreducible if the only subspaces of $\mathbb{C}^2$ invariant under $\rho(G)$ are $\{0\}$ and $\mathbb{C}^2$ \cite{shalen}. This is equivalent to saying that $\rho$ can't be conjugated to a representation by upper triangular matrices. Otherwise $\rho$ is called reducible. We will call a character irreducible (reducible) if the corresponding representation is irreducible (reducible).

First, I determine which components of $X(\mathcal{M})$ contain characters of irreducible representations.
Computation with SnapPy \cite{SnapPy} shows that the Alexander polynomial $\Delta_{\mathcal{M}}$ of $m137$ is $1$, which has no root. So there are no reducible non-abelian representations \cite[Section 6.1]{metabelian}. Therefore all the reducible representations are abelian.
Since $H_1(\mathcal{M})=\mathbb{Z}$, there is only one such component and it is parameterized by the image of $\beta$ and is isomorphic to Hom$(\mathbb{Z}, SL_2(\mathbb{C}))//SL_2(\mathbb{C})\simeq \mathbb{C}$.  
Moreover, it is disjoint from any component of $X(\mathcal{M})$ containing the character of an irreducible representation \cite[Section 6.2]{metabelian}.
For more details, we refer the readers to Tillmann's note \cite{Tillmann} where he studied $m137$ as an example.

An abelian representation of $\pi_1(\mathcal{M})$ that induces an abelian representation of $\pi_1(\mathcal{M}(1,n))$ factors through the abelianization $ab(\pi_1(\mathcal{M}(1,n)))=1$. So they correspond to trivial $SL_2(\mathbb{C})$ representations and we don't need to worry about them.

Now we consider components of $X(\mathcal{M})$ that contain the character of an irreducible representation. We have:
\begin{lemma}\label{charlemma}
There is a single component $X_0(\mathcal{M})$ of $X(\mathcal{M})$ containing an irreducible character. The functions $s=tr \rho(\lambda)=tr \rho(\alpha^{-1}\beta^{-1})=tr \rho(\alpha\beta)$ and $t=tr \rho(\beta)$ give complete coordinates on $X_0(\mathcal{M})$, which is the curve in $\mathbb{C}^2$ cut out by
\begin{equation*}
(-2 - 3s + s^3)t^4 + (4 + 4s - s^2 - s^3)t^2  -1 = 0
\end{equation*}
Moreover, $w:=tr \rho(\lambda\beta)=tr \rho((\lambda\beta)^{-1})=t-\frac{1}{t(s+1)}$.
\end{lemma}
\begin{proof}[Proof of Lemma \ref{charlemma}]
Let $X_0(\mathcal{M})$ be $X(\mathcal{M})-\{\text{reducible characters}\}$. From the discussion above, we know that all the reducible characters form a single component of $X(\mathcal{M})$ and this component is disjoint from any other component of $X(\mathcal{M})$. So $X_0(\mathcal{M})$ is Zariski Closed. We will show later that $X_0(\mathcal{M})$ is actually an irreducible algebraic variety, as claimed in the lemma.

Suppose $[\rho]\in X_0(\mathcal{M})$. So $\rho$ is an irreducible representation. By conjugating $\rho$ if necessary, we can assume that $\rho$ has the form
\begin{displaymath}
\rho(\lambda)=
\begin{pmatrix}
 z&1\\
 0&1/z
\end{pmatrix}
,\quad
\rho(\beta)=
\begin{pmatrix}
 x&0\\
 y& 1/x
\end{pmatrix}.
\end{displaymath}
From the relator of $\pi_1(\mathcal{M})$ in \eqref{pi1} we have $\rho(\beta)^{-1}\rho(\lambda)^{-1}\beta^{-1}\rho(\lambda)^{-1}\rho(\beta)^2\rho(\lambda)=\rho(\lambda)\rho(\beta)^{-2}\rho(\lambda)^{-1}\rho(\beta)^2$. Comparing the entries of the matrices on both sides, we get four equations. These four equations together with $s=z+1/z$, $t=x+1/x$ and $w=zx+z^{-1}x^{-1}+y$ form a system $\mathcal{S}$ which defines $X_0(\mathcal{M})$. 
By computing a Gr\"{o}bner basis of this system, SageMath \cite{sage} gives the following generators of the radical ideal $I=I(X_0(\mathcal{M}))$:

\begin{align}
&stw - t^2 - w^2 - s + 2 \label{1}\\
&t^3 - w^3 + st - sw - 2t + w\\
&st^2 - tw - w^2 - s + 1\label{3}\\
&sw^3 - s^2t + s^2w - t^2w - tw^2 + st - sw + t
\end{align}

Subtracting \eqref{3} from \eqref{1}, we get:
\begin{align}\label{wst}
w=t-\frac{1}{t(s+1)}.
\end{align}
Eliminating $w$, we get a defining equation for $X_0(\mathcal{M})$:
\begin{equation}\label{char}
\begin{split}
0 &= (-2 - 3s + s^3)t^4 + (4 + 4s - s^2 - s^3)t^2 -1\\
&= (s - 2)(s + 1)^2t^4 - (s - 2)(s + 2)(s + 1)t^2 -1.
\end{split}
\end{equation}
Thus, we can think of $X_0(\mathcal{M})$ as living in $\mathbb{C}^2$.

To prove the lemma, we must show that $X_0(M)$ is irreducible or equivalently the polynomial $P(s,t):=(s - 2)(s + 1)^2t^4 - (s - 2)(s + 2)(s + 1)t^2 -1$ in \eqref{char} does not factor in $\mathbb{C}[s,t]$. Assume $P(s,t)$ factors. Suppose it factors as
\begin{displaymath}
(at^2+bt+c)(dt^2+et-1/c)=adt^4+(ae+bd)t^3+(cd-a/c+be)t^2+(ce-b/c)t-1,
\end{displaymath} where $a,b,d,e\in \mathbb{C}[s]$ and $c\in \mathbb{C}-\{0\}$.
Setting the coefficients of $t$ and $t^3$ to be $0$, we get $b=c^2e$ and $ae=-c^2de$. If $e\neq 0$, then $a=-c^2d$. But this is impossible as $ad=(s - 2)(s + 1)^2$ is a polynomial in $s$ of odd degree. So $e=0$ and it follows that $b=0$. Comparing the coefficients of $t^2$ and $t^4$, we get
\begin{equation}\label{eq3}
ad=(s-2)(s+1)^2
\end{equation}
and
\begin{equation}\label{eq4}
cd-a/c=-(s-2)(s+2)(s+1).
\end{equation}
So $\text{degree}(a)+\text{degree}(d)=3$ and $\text{max}\{\text{degree}(a), \text{degree}(d)\}\geq 3$, which implies exactly one of $a$ and $d$ has degree $3$ and the other has degree $0$. Without loss of generality, we can assume that $\text{degree}(a)=3$ and $\text{degree}(d)=0$. Multiply both sides of \eqref{eq4} by $c$, we get $a=c^2d+c(s-2)(s+2)(s+1)$. So the coefficient of $s^3$ in $a$ is $c$. Comparing with the coefficient of $s^3$ in \eqref{eq3}, we know that $d=1/c$. Eliminating $a$ and $d$ gives us an equality $1+(s-2)(s+2)(s+1)=(s-2)(s+1)^2$, which does not hold.

Else suppose $P(s,t)$ factors as
\begin{displaymath}
(at+c)(bt^3+dt^2+et-1/c)=abt^4+(ad+cb)t^3+(cd+ae)t^2+(ce-a/c)t-1,
\end{displaymath}
where $a,b,d,e\in \mathbb{C}[s]$ and $c\in \mathbb{C}-\{0\}$. Setting the coefficients of $t$ and $t^3$ to be $0$, we get $a=c^2e$ and $b=ced$. Comparing the coefficients of $t^2$ and $t^4$, we get
\begin{equation}\label{eq1}
c^3de^2=(s-2)(s+1)^2
\end{equation}
and
\begin{equation}\label{eq2}
cd+c^2e^2=-(s-2)(s+2)(s+1).
\end{equation}
So $\text{degree}(d)+2\text{degree}(e)=3$ and $\text{max}\{\text{degree}(d), 2\text{degree}(e)\}\geq 3$ which implies $\text{degree}(d)=3$ and $\text{degree}(e)=0$. Comparing the coefficients of $s^3$ in \eqref{eq1} and \eqref{eq2}, we know that $c^2e^2=-1$. Plugging into \eqref{eq1}, we get $cd=(s-2)(s+1)^2$, which when plugging into \eqref{eq2} implies $c^2e^2=-(s+1)(s-2)$, a contradiction.
So $P(s,t)$ is irreducible over $\mathbb{C}$. Therefore $X_0(\mathcal{M})$ has only one component.
\end{proof}

To find irreducible $SL_2(\mathbb{R})$ representations of $\pi_1(\mathcal{M})$, we need to check all real points on $X_0(\mathcal{M})$, which correspond to real solutions of \eqref{char}.
Notice that equation \eqref{char} has no solutions when $s=-1 \text{ or } 2$, so \eqref{char} is a quadratic equation in $t^2$. In order for $t$ to be real, $t^2$ has to be real and nonnegative. Then first we need the discriminant to be nonnegative. That is:
\begin{displaymath}
\Delta_1=(s+1)^2(s-2)(s^3+2s^2-4s-4)\geq 0.
\end{displaymath}
So $s\in U:=(-\infty, p_1]\cup[p_2, p_3]\cup(2, \infty)$, where $p_1\approx -2.9032$, $p_2\approx -0.8061$ and $p_3\approx 1.7093$ are three roots of cubic polynomial $s^3+2s^2-4s-4$.

The following lemma will help us determine when a $SL_2(\mathbb{C})$ representation of $\pi_1(\mathcal{M})$ can be conjugated into $SL_2(\mathbb{R})$ by simply checking where it lies on the character variety.
\begin{lemma}\label{main lemma}
The real points $X_{0,\mathbb{R}}(\mathcal{M})=X_0(\mathcal{M})\cap \mathbb{R}^2$ of $X_0(\mathcal{M})$ has $6$ connected components:

Points on the two components with $|s|<2$ correspond to $SU_2(\mathbb{C})$ representations.

Points on the four components with $|s|>2$ correspond to $SL_2(\mathbb{R})$ representations. 
\end{lemma}

\begin{remark}
The above lemma shows that in our case, the absolute value of one character being smaller than $2$ implies that the representation is $SU_2(\mathbb{C})$. But in general, this is not true.
\end{remark}

To prove this lemma, we need to determine when $[\rho]\in X_{0,\mathbb{R}}(\mathcal{M})$ corresponds to $\rho \in SU_2(\mathbb{C})$ and when it corresponds to $\rho \in SL_2(\mathbb{R})$. It can't be in both because otherwise it would be reducible \cite[Lemma 2.10]{CD16} and we know $X_0(\mathcal{M})$ contains only irreducible characters.
The tool we use is a reformulation of Proposition 3.1 in \cite{SU} which states that given three angles $\theta_i\in[0,\pi]$, $i=1,2,3$, there exist three $SU_2(\mathbb{C})$ matrices $C_i$, satisfying $C_1 C_2 C_3 =I$ with eigenvalues $\exp(\pm i
\theta_i)$ respectively if and only if these angles satisfy:
\begin{align}\label{ineq}
|\theta_1-\theta_2| \leq \theta_3 \leq \min\{\theta_1+\theta_2, 2\pi - (\theta_1+\theta_2) \}.
\end{align}
We want to rewrite the above inequality in terms of traces of $C_1, C_2$ and $C_3$. We have the following lemma:

\begin{lemma}\label{su2c}
Suppose $t_1, t_2, t_3\in (-2,2)$ are the traces of three matrices $C_1, C_2, C_3\in SL_2(\mathbb{C})$ satisfying $C_1C_2C_3=I$. Then $C_1, C_2, C_3$ are simultaneously conjugate into $SU_2(\mathbb{C})$ if and only if
\begin{align*}
(2t_3-t_1t_2)^2\leq (4-t_1^2)(4-t_2^2).
\end{align*}
\end{lemma}

\begin{proof}

Suppose $t_1=2\cos(\theta_1)$, $t_2=2\cos(\theta_2)$ and $t_3=2\cos(\theta_3)$ where $\theta_1, \theta_2, \theta_3\in [0,\pi]$.

If $0\leq\theta_1+\theta_2\leq \pi$, then the inequality \eqref{ineq} becomes $|\theta_1-\theta_2| \leq \theta_3 \leq \theta_1+\theta_2$. Taking cosine, we get $\cos(\theta_1+\theta_2) \leq \cos(\theta_3) \leq \cos(\theta_1-\theta_2)$.

If $\pi\leq\theta_1+\theta_2\leq 2\pi$, then the inequality becomes $|\theta_1-\theta_2| \leq \theta_3 \leq 2\pi-(\theta_1+\theta_2)$. Taking cosine, we also get $\cos(\theta_1+\theta_2) \leq \cos(\theta_3) \leq \cos(\theta_1-\theta_2)$.

Use the relations $t_1=2\cos(\theta_1)$, $t_2=2\cos(\theta_2)$, and $t_3=2\cos(\theta_3)$, we get in both cases that:
\begin{displaymath}
\frac{t_1t_2}{4}-\sqrt{\left(1-\frac{t_1^2}{4}\right)\left(1-\frac{t_2^2}{4}\right)}\leq \frac{t_3}{2}\leq \frac{t_1t_2}{4}+\sqrt{\left(1-\frac{t_1^2}{4}\right)\left(1-\frac{t_2^2}{4}\right)}.
\end{displaymath}
Then
\begin{displaymath}
-\sqrt{\left(1-\frac{t_1^2}{4}\right)\left(1-\frac{t_2^2}{4}\right)}\leq \frac{t_3}{2}-\frac{t_1t_2}{4}\leq \sqrt{\left(1-\frac{t_1^2}{4}\right)\left(1-\frac{t_2^2}{4}\right)}.
\end{displaymath}
So we have:
\begin{align*}
\left|\frac{t_3}{2}-\frac{t_1t_2}{4}\right|\leq \sqrt{\left(1-\frac{t_1^2}{4}\right)\left(1-\frac{t_2^2}{4}\right)}.
\end{align*}
Squaring both sides and simplifying, we get
\begin{align*}
(2t_3-t_1t_2)^2\leq (4-t_1^2)(4-t^2),
\end{align*}
as desired.
\end{proof}

With the criterion of Lemma \ref{su2c} in hand, we now can prove Lemma \ref{main lemma}.
\begin{proof}[Proof of Lemma \ref{main lemma}]

The six components correspond to $s\in (-\infty, p_1]\cup[p_2, p_3]\cup(2, \infty)$ and $t\in (-\infty,0)\cup(0, \infty)$.

Set $C_1=\rho(\lambda)$, $C_2=\rho(\beta)$ and $C_3=\rho(\beta^{-1}\lambda^{-1})=\rho((\lambda \beta)^{-1})$. Then $t_1=s$, $t_2=t$ and $t_3=w$. Applying Lemma \ref{su2c} we have:
\begin{align}\label{above}
(2w-st)^2\leq (4-s^2)(4-t^2).
\end{align}
Plugging \eqref{wst} into \eqref{above} and simplifying:
\begin{align*}
(s-2)^2t^2+\frac{4(s-2)}{s+1}+\frac{4}{t^2(s+1)^2}\leq (4-s^2)(4-t^2).
\end{align*}
Multiplying both sides by $t^2(s+1)^2$, we get:
\begin{align*}
(s+1)^2(s-2)^2t^4+4(s-2)(s+1)t^2+4\leq (4-s^2)(s+1)^2(4-t^2)t^2.
\end{align*}
which simplifies to:
\begin{align*}
-(s + 1)^2(s - 2)t^4 + (s^2+3s+3)(s - 2)(s + 1)t^2 + 1\leq 0.
\end{align*}
Plugging in \eqref{char}, we get
\begin{align*}
(s+1)^3(s-2)t^2\leq 0,
\end{align*}
which always holds when $s\in(p_2\approx-0.8061,p_3\approx1.7093)\subset(-2,2)$.

So, points on $X_{0,\mathbb{R}}(\mathcal{M})$ correspond to $SU_2(\mathbb{C})$ representations if and only if $|s|<2$ and correspond to $SL_2(\mathbb{R})$ representations if and only if $|s|>2$.
\end{proof}


\begin{proof}[Proof of Theorem \ref{result} (a)]
Lemma 2 tells us a $SL_2(\mathbb{C})$ representation $\rho$ of $m137$ is real if and only if eigenvalues of $\rho(\lambda)$ are real. Moreover, the condition $\mu\lambda^n=1$ forces the eigenvalues of $\rho(\mu)$ to also be real in this case. So we could restrict our attention to $|s|>2$ and look at the A-polynomial instead (see e.g. \cite{metabelian} for definition of A-polynomial). Recall that $z$ is an eigenvalue of $\rho(\lambda)$. Denote by $m$ the eigenvalue of $\rho(\mu)$ which shares the same eigenvector with $z$.
The A-polynomial of $m137$ is computed by SAGE \cite{sage} as:
\begin{equation*}
\begin{split}
&(z^4+2z^5+3z^6+z^7-z^8-3z^9-2z^{10}-z^{11})+m^2(-1-3z-2z^2-z^3\\
&+2z^4+4z^5+z^6+4z^7+z^8+4z^9+2z^{10}-z^{11}-2z^{12}-3z^{13}-z^{14})\\
&+m^4(-z^3-2z^4-3z^5-z^6+z^7+3z^8+2z^9+z^{10}).
\end{split}
\end{equation*}
Denote by $A=-1-2z-3z^2-z^3+z^4+3z^5+2z^6+z^7=(z-1)(z^2+z+1)^3$ and
$B=1+3z+2z^2+z^3-2z^4-4z^5-z^6-4z^7-z^8-4z^9-2z^{10}+z^{11}+2z^{12}+3z^{13}+z^{14}$.
So the A-polynomial could be simplified as $-z^4A-Bm^2+z^{3}Am^4$. We are interested in the real solutions of
\begin{equation}\label{Apoly2}
-z^4A-Bm^2+z^{3}Am^4=0.
\end{equation}

Now consider the $(1, n)$ Dehn filling on $m137$. Then we are adding an extra relation $\mu\lambda^n=1$, which is $\rho(\mu)\rho(\lambda)^n=I$ under the representation $\rho$, i.e.
\begin{align*}
\rho(\mu)=\rho(\lambda)^{-n}=
\begin{pmatrix}
 z^{-n} &  * \\ 0 &  z^n
\end{pmatrix}.
\end{align*}
Restricting to $\partial \mathcal{M}$ gives us the relation $m=z^{-n}$.

When $n$ is negative, we shall denote $n'=-n$. So we have $m=z^{n'}$.
Plugging into \eqref{Apoly2} and dividing both sides by $z^4$, we get
\begin{equation}\label{Apoly3}
-A-Bz^{2n'-4}+Az^{4n'-1}=0.
\end{equation}
We will show the following lemma is true, completing the proof of Theorem \ref{result} (a).
\begin{lemma}\label{apolylemma}
Equation \eqref{Apoly3} has no real solutions when $n'$ is large enough.
\end{lemma}
\begin{proof}[Proof of Lemma \ref{apolylemma}]
Define $F(z)=A(z^{4n'-1}-1)-Bz^{2n'-4}$. We'll show $F(z)>0$.

First notice that $A=0$ only when $z=1$. And $A>0$ when $z>1$ while $A<0$ when $z<1$.
The polynomial $B$ has 6 real roots which are all simple: $-2.3396$, $-1.4121$, $-0.7082$, $-0.4274$, $0.8684$, $1.1516$ (rounded to the fourth digit).

As we saw earlier, the domain for $s$ is $U:=(-\infty, p_1\approx -2.9032]\cup[p_2\approx -0.8061, p_3\approx 1.7093]\cup(2, \infty)$. So the $|s|>2$ condition restricts $s$ to $(-\infty, p_1\approx -2.9032]\cup(2, \infty)$. Then $z\in V:=(-\infty, -2.5038]\cup[-0.3994, 0)\cup(0, 1)\cup(1,\infty)$.
Notice that $z^7A(1/z)=-A(z)$ and $z^{14}B(1/z)=B(z)$. Interchange $z$ with $1/z$ in $F(z)$ gives us $F(1/z)=A(1/z)(z^{-(4n'-1)}-1)-B(1/z)z^{-(2n'-4)}=F(z)/z^{4n'+6}$. So we can assume $|z|<1$.

\subsubsection*{case 1: $0.8684\leq z<1$}
In this case, we have $A(z)<0$, $B(z)\leq 0$ and $z^{4n'-1}-1<0$. So $F(z)>0$.

\subsubsection*{case 2: $-0.3994\leq z<0.8684$ and $z\neq 0$}
In this case, we have $A(z)<C_5<0$ and $C_6>B(z)>0$ for some constants $C_5$ and $C_6$. When $n'$ is large enough, we have $|C_5|\times|(z^{4n'-1}-1)|>C_6z^{2n'-4}$. So $A(z^{4n'-1}-1)=|A|\times|(z^{4n'-1}-1)|>Bz^{2n'-4}$ and it follows that $F(z)>0$.

Therefore when $n'=-n$ is large enough, we always have $F(z)>0$ on the domain $V$. So equation \eqref{Apoly3} has no real solution when $n'\gg 0$.
\end{proof}
It follows from the above lemma that equation \eqref{Apoly2} has no real solution when $n\ll 0$ and thus equality $\rho(\mu)\rho(\lambda)^n=I$ does not hold for $n\ll 0$.

From all the discussion above, we can now conclude that 
$\mathcal{M}(1,n)$ has no nontrivial $SL_2(\mathbb{R})$ representation and thus no nontrivial $PSL_2(\mathbb{R})$ representation for $n\ll 0$. Since the first Betti number of $\mathcal{M}(1,n)$ is 0, the lift of a trivial $PSL_2(\mathbb{R})$ representation of $\pi_1(\mathcal{M}(1,n))$ into $\widetilde{PSL_2}(\mathbb{R})$ will be trivial. So all representations of $\pi_1(\mathcal{M}(1,n))$ into $\widetilde{PSL_2}(\mathbb{R})$ are trivial for $n\ll 0$, proving Theorem \ref{result} (a).
\end{proof}

In contrast, when $n$ is positive there are examples of non trivial $SL_2(\mathbb{R})$ representations.

Plugging $m=z^{-n}$ into \eqref{Apoly2} and multiplying both sides by $z^{4n-3}$, we get
\begin{equation*}
-A+Bz^{2n-3}+Az^{4n+1}=0.
\end{equation*}
Similarity, define $G(z)=A(z^{4n+1}-1)+Bz^{2n-3}$.
Since $G(1)=-4$, $G(0.8684)>0$, $G(z)$ must have at least one root in $[0.8684, 1)$. So $\pi_1(\mathcal{M}(1,n))$ has at least one nontrivial $SL_2(\mathbb{R})$ representation for any $n>0$. They lift to a $\widetilde{PSL_2}(\mathbb{R})$ representations, since the Euler number of any representation of an integral homology sphere vanishes \cite[Section 6]{Ghys}.

\section{No L-space fillings}
In this section, I will prove Theorem \ref{result} (b) using results from Gillespie's paper \cite{gillespie}, which is based on Rasmussen and Rasmussen's paper \cite{simple}. In fact, I will show that none of the non-longitudinal fillings of $m137$ is an L-space. The homology groups in this section are all homology with integral coefficients.

Suppose $Y$ is a compact connected $3$-manifold with a single torus as boundary. I will follow Gillespie's \cite{simple} notation. Define the set of slopes on $\partial Y$ as:
\begin{displaymath}
\mathcal{S}l(Y)=\{a\in H_1(\partial Y)|\ a\text{ is primitive}\}/\pm 1.
\end{displaymath}
Define the set of L-space filling slopes of $Y$:
\begin{displaymath}
\mathcal{L}(Y)=\{a\in \mathcal{S}l(Y)|\ Y(a) \text{ is an L-space}\}.
\end{displaymath}
Moreover, $Y$ is said to have genus $0$ if $H_2(Y, \partial Y)$ is generated by a surface of genus $0$.

We will use Theorem 1.2 from Gillespie's paper \cite{gillespie} which is stated as:
\begin{theorem}
\label{thm1}
The following are equivalent
\begin{itemize}
\item[1)] $\mathcal{L}(Y)= Sl(Y)-\{l\}$.
\item[2)] $Y$ has genus $0$ and has an L-space filling.
\end{itemize}
\end{theorem}
\begin{proof}[Proof of Theorem \ref{result} (b)]
Let $l\in Sl(\mathcal{M})$ be the homological longitude.
In our case $l$ can be taken to be $[\lambda]$.
I will show that none of the $(1,n)$ fillings to $\mathcal{M}$ is an L-space.

I will find one non L-space filling first. Snappy \cite{SnapPy} shows that $(1,-1)$ filling on the knot $8_{20}$ complement with homological framing is homeomorphic to $m 011(2,3)$, which is also homeomorphic to $\mathcal{M}(1,-3)$. 
Ozsv\'{a}th and Szab\'{o} showed that if some $(1, p)$ Dehn filling of a knot complement in $S^3$with homological framing is an L-space, then the Alexander polynomial of the knot has coefficients $\pm 1$ \cite[Corollary 1.3]{lspaceknot}. We can compute with SnapPy \cite{SnapPy} that the Alexander Polynomial of $8_{20}$ is $x^4 - 2x^3 + 3x^2 - 2x + 1$. So $\mathcal{M}(1,-3)$ is not an L-space.  
Therefore
\begin{displaymath}
-3l+[\mu]\notin \mathcal{L}(\mathcal{M}) \neq \mathcal{S}l(\mathcal{M})-\{l\}\ni -3l+[\mu] ,
\end{displaymath}
By Theorem \ref{thm1}, either $\mathcal{M}$ has no L-space fillings or $\mathcal{M}$ has positive genus.

The manifold $\mathcal{M}$ can be viewed as the complement of a knot $K$ in $S^2\times S^1$ \cite{Nathan_example}. This knot $K$ intersects each $S^2$ three times. So $[K]\neq 0$ in $H_1(S^2\times S^1;\mathbb{Z})$. It follows that $H_2(\mathcal{M},\partial \mathcal{M})$ is generated by genus $0$ surface $(S^2\times \{P\})\cap \mathcal{M}$ for generic point $P$ on $K$. So $\mathcal{M}$ has genus $0$, which forces $\mathcal{M}$ to have no L-space filling. Therefore none of the integral homology spheres $\mathcal{M}(1,n)$ is an L-space.
\end{proof}

\bibliography{m137}
\bibliographystyle{plain}
\end{document}